\documentclass[12pt,reqno]{amsart}
\usepackage{amscd,amsmath,amsthm,amssymb}
\usepackage{color}
\usepackage{tikz}
\usepackage{url}

 \newtheorem{Theorem}{Theorem}[section]
 \newtheorem{Lemma}[Theorem]{Lemma}
 
 \newtheorem{Proposition}[Theorem]{Proposition}

 \newtheorem{Definition}[Theorem]{Definition}

 \def\qed{\ifhmode\textqed\fi
       \ifmmode\ifinner\quad\qedsymbol\else\dispqed\fi\fi}
 \def\textqed{\unskip\nobreak\penalty50
        \hskip2em\hbox{}\nobreak\hfill\qedsymbol
        \parfillskip=0pt \finalhyphendemerits=0}
 \def\dispqed{\rlap{\qquad\qedsymbol}}

\def\ZZ{\mathbb{Z}}
\def\m{\mathfrak{m}}

\def\height{\textup{height}}
\def\Ass{\textup{Ass}}

\def\depth{\textup{depth\,}}

\def\supp{\textup{supp}}
\def\G{\mathcal{G}}
 
\begin{document}

\title{Edge ideals whose all matching powers\\ are bi-Cohen-Macaulay}
\author{Marilena Crupi, Antonino Ficarra}

\address{Marilena Crupi, Department of mathematics and computer sciences, physics and earth sciences, University of Messina, Viale Ferdinando Stagno d'Alcontres 31, 98166 Messina, Italy}
\email{mcrupi@unime.it}

\address{Antonino Ficarra, BCAM -- Basque Center for Applied Mathematics, Mazarredo 14, 48009 Bilbao, Basque Country -- Spain, Ikerbasque, Basque Foundation for Science, Plaza Euskadi 5, 48009 Bilbao, Basque Country -- Spain}
\email{aficarra@bcamath.org,\,\,\,\,\,\,\,\,\,\,\,\,\,antficarra@unime.it}

\subjclass[2020]{Primary 13C05, 13C14, 13C15; Secondary 05E40}
\keywords{Bi-Cohen–Macaulay ideals, Complete graphs, Paths, Generic graphs, Squarefree powers, Matching powers.}

\begin{abstract}
We classify all graphs $G$ satisfying the property that all matching powers $I(G)^{[k]}$ of the edge ideal $I(G)$ are bi-Cohen-Macaulay for $1\le k\le\nu(G)$, where $\nu(G)$ is the maximum size of a matching of $G$.
\end{abstract}

\maketitle
\section*{Introduction}

In \cite{FV2005}, Fl\o ystad and Vatne  introduced the concept of bi-Cohen-Macaulay simplicial complex.  A simplicial complex $\Delta$ is called bi-Cohen-Macaulay, if $\Delta$ and its Alexander dual
$\Delta^\vee$ are Cohen-Macaulay. In that paper the authors associated to each simplicial complex $\Delta$, in a natural way, a complex of coherent sheaves and showed that this complex reduces to a coherent sheaf if and only
if $\Delta$ is bi-Cohen-Macaulay. Such a notion has suggested the definition of bi-Cohen-Macaulay squarefree monomial ideal. 

Let $S=K[x_1,\dots,x_n]$ be the standard graded polynomial ring over a field $K$ and let $I\subset S$ be a squarefree monomial ideal. We say that $I$ is Cohen-Macaulay if $S/I$ is a Cohen-Macaulay ring. Recall that $I$ can be considered as the Stanley-Reisner ideal of a simplicial complex on the vertex set $[n]= \{1,\dots,n\}$. Attached to $I$ is the Alexander dual $I^\vee$, which is again a squarefree monomial ideal. We say that $I$ is \textit{bi-Cohen-Macaulay} (bi-CM, for short) if both $I$ and $I^\vee$ are Cohen-Macaulay ideals. By the Eagon-Reiner criterion \cite[Theorem 8.1.9]{HHBook} $I$ has a linear resolution if and only if $I^\vee$ is Cohen-Macaulay. Hence, $I$ is bi-CM if and only if it is Cohen-Macaulay with linear resolution.

Such a notion can be revisited in graph theory. More in detail, let $G$ be a finite simple graph on the vertex set $[n]= \{1,\dots,n\}$ and let $I(G)$ be the edge ideal of $G$, that is, the squarefree monomial ideal of $S$ whose generators are the monomials $x_ix_j$ of the polynomial ring $S=K[x_1,\dots,x_n]$ with $\{i,j\}$ an edge of $G$. We say that $G$ is bi-CM if $I(G)$ is bi-CM. By the Eagon-Reiner criterion, previously mentioned, it follows that a bi-CM graph $G$ is connected. Indeed, if this is not the case, then there exist induced subgraphs $G_1$ and $G_2$ of $G$ such that the vertex set $V (G)$ is the disjoint union of $V(G_1)$ and $V(G_2)$. It follows that $I(G) = I(G_1)+I(G_2)$, and the ideals $I(G_1)$ and $I(G_2)$ are ideals in a different set of variables. Therefore, the free resolution of $S/I(G)$ is obtained as the tensor product of the resolutions of $S/I(G_1)$ and $S/I(G_2)$. Thus $I(G)$ has relations of degree $4$, so that $I(G)$ does not have a linear resolution.

In the last years many authors have tried to classify all bi-CM graphs. The pioneer paper is \cite{HR16}, where the authors gave a complete classification of bi-CM graphs, up to separation, and in the case they are bipartite or chordal. 

In this article we classify the graphs $G$ which satisfy the following property: \textit{all matching powers $I(G)^{[k]}$ of  the edge ideal $I(G)$ of $G$ are bi-CM for all $1\le k\le\nu(G)$, where $\nu(G)$ is the maximum size of a matching of $G$.} 

This question has been inspired by many recent articles in which special classes of graphs, whose matching powers of their edge ideals are Cohen-Macaulay, have been considered.  In \cite{DRS24}, Das, Roy and Saha proved that $I(G)^{[k]}$ is Cohen-Macaulay for all $1 \le k \le  \nu(G)$, if $G$ is a Cohen-Macaulay forest, and recently in \cite{FM}, Ficarra and Moradi, have proved that all matching powers of the edge ideal of a chordal graph $G$ are Cohen-Macaulay if and only if $G$ is either a complete graph or a Cohen-Macaulay forest. Furthermore, they have proved that all matching powers of the edge
ideal of a Cameron-Walker graph $G$ are Cohen-Macaulay if and only if $G$ is a complete graph on $2$ or $3$ vertices.

In the present  paper we prove that all matching powers of the edge ideal of a finite, simple graph $G$ on $n$ non-isolated vertices are bi-CM if and only if $G$ is the complete graph $K_n$ or the complementary graph of a path $P_n$ on $n$ vertices. Our main tool is the notion of \textit{vertex splittable} ideal introduced in \cite[Definition 2.1]{MKA16}.

The paper is structured as follows. Section \ref{sec1} contains some notions and results useful for the development of the topic. We deeply discuss the notion of matching powers of the edge ideal of a graph and its relations with the concept of squarefree powers of a squarefree monomial ideal. The main result in the section is Proposition \ref{Prop:KP} that states, via the notion of principal ${\bf t}$-spread Borel ideal, that all matching powers of the edge ideal of a graph $G$ are Cohen-Macaulay if $G$ is the complete graph $K_n$  or the complementary graph of a path $P_n$ on $n$ vertices, for $n\ge 4$. This result is reversed in Section \ref{sec2}.

Section \ref{sec2} contains our main result (Theorem \ref{thm:main}) that states the classification we are looking for. We prove that if $G$ is a finite, simple graph on $n\ge4$ non-isolated vertices, then $I(G)^{[k]}$ is bi-CM for all $1\le k\le\nu(G)$ if and only either $G=K_n$ or $G\cong P_n^c$, where $P_n^c$ is the complementary graph of a path $P_n$ on $n$ vertices. A key result is Lemma \ref{Lem:notCM} that is a criterion for determine if a squarefree monomial ideal is Cohen-Macaulay.

\section{Auxiliary notions and results}\label{sec1}
In this section we discuss some notions and results useful for the development of the paper.

Throughout the article the graphs $G$ considered will all be finite, simple graphs, that is, they will
have no double edges and no loops. Furthermore, we assume that $G$ has no isolated
vertices. The vertex set of $G$ will be denoted by $V (G)$ and we will assume that $V(G)=[n]=\{1, \ldots, n\}$, unless otherwise stated. The set of edges of $G$ will be denoted by $E(G)$. 

We say that $G$ is the \textit{complete graph} on $n$ vertices if $E(G)=\{\{i,j\}:1\le i<j\le n\}$, whereas, we say that $G$ is a \textit{path} on $n$ vertices if, up to a relabeling, we have $E(G)=\{\{1,2\},\{2,3\},\dots,\{n-1,n\}\}$. As usually, a complete graph on $n$ vertices will be denoted by $K_n$ and a path on $n$ vertices will be denoted by $P_n$.

Recall that the complementary graph of a graph $G$ is the graph $G^c$ with vertex set $V(G^c)=V(G)$ and with edge set $E(G^c)=E(K_n)\setminus E(G)$.

A $k$-\textit{matching} of $G$ is a subset $M$ of $E(G)$ of size $k$ such that $e\cap e'=\emptyset$ for all $e,e'\in M$ with $e\ne e'$. We denote by $V(M)$ the vertex set of $M$, that is, the set $\{i\in V(G):i\in e\ \text{for}\ e\in M\}$.  The \textit{matching number} of $G$, denoted by $\nu(G)$, is the maximum size of a matching of $G$. 

We say that a graph $H$ is a \textit{subgraph of $G$} if $V(H)\subseteq V(G)$ and $E(H)\subseteq E(G)$. A subgraph $H$ of $G$ is said to be an \textit{induced subgraph} if for any two vertices $i, j$ in $H$, $\{i, j\}\in E(H)$ if and only if $\{i, j\}\in E(G)$. If $A$ is a subset of $V(G)$, the \textit{induced subgraph} on $A$ is the graph with vertex set $A$ and the edge set $\{\{i, j\}: \mbox{$i, j \in A$ and $\{i, j\}\in E(G)$}\}$.

Now let $S=K[x_1,\dots,x_n]$ be the standard graded polynomial ring over a field $K$. For a non-empty subset $A$ of $[n]$, we set ${\bf x}_A=\prod_{i\in A}x_i$.  

Let $1\le k\le\nu(G)$. We denote by $I(G)^{[k]}$ the squarefree monomial ideal generated by ${\bf x}_{V(M)}$ for all $k$-matchings $M$ of $G$. We call $I(G)^{[k]}$ the \textit{matching power} of $I(G)$. If $k=1$, then $I(G)^{[1]}$ is the well-known ideal, called  the \textit{edge ideal} of $G$ \cite{RVbook}, and we denote it simply by $I(G)=(x_ix_j : \{i,j\}\in E(G))$.

There is a connection of such a notion with the concept of squarefree power (see, for instance, \cite{BHZ}) of a squarefree monomial ideal of $S$. Let $I\subset S$ be a squarefree monomial ideal and $\G(I)$ be its unique minimal set of monomial generators. The \textit{$k$th squarefree power} of $I$, denoted by $I^{[k]}$, is the ideal generated by the squarefree monomials of $I^k$. Thus $u_1u_2\cdots u_k$, $u_i\in \G(I)$, $i\in[k]$, belongs to $\G(I^{[k]})$ if and only if $u_1,u_2,\dots,u_k$ is a regular sequence. Let $\nu(I)$ be the \textit{monomial grade} of $I$, \emph{i.e.}, the maximum among the lengths of a monomial regular sequence contained in $I$. Then $I^{[k]}\ne(0)$ if and only if $k\le\nu(I)$. Hence, the ideal $I(G)^{[k]}$ is the $k$th squarefree power of $I(G)$ and $\nu(I(G))=\nu(G)$. 

See also \cite{BHZ,CFL3,DRS24,DRS25,EF,EF1,EH2021,EHHM2022a,EHHM2022b,FPack2,FHH2022,FM,KNQ24,SASF2024,SASF2022,SASF2023,SASF2024b} for further studies on squarefree and matching powers.

The following result is a consequence of \cite[Corollary 1.3]{EHHM2022a}.

\begin{Lemma}\label{Lem:res}
	Let $G$ be a graph and let $H$ be an induced subgraph of $G$. If $I(G)^{[k]}$ has linear resolution, then $I(H)^{[k]}$ has linear resolution too.
\end{Lemma}

Following \cite[Definition 2.1]{MKA16}, a monomial ideal $I\subset S$ is called \textit{vertex splittable} if it can be obtained by the following recursive procedure.
\begin{enumerate}
	\item[\rm (i)] If $u$ is a monomial and $I=(u)$, $I=0$ or $I=S$, then $I$ is vertex splittable.\smallskip
	\item[\rm (ii)] If there exists a variable $x_i$ and vertex splittable ideals $I_1\subset S$ and $I_2\subset K[x_1,\dots,x_{i-1},x_{i+1},\dots,x_n]$ such that $I=x_iI_1+I_2$, $I_2\subseteq I_1$ and $\mathcal{G}(I)$ is the disjoint union of $\mathcal{G}(x_iI_1)$ and $\mathcal{G}(I_2)$, then $I$ is vertex splittable.
\end{enumerate}
In the case (ii), the decomposition $I=x_iI_1+I_2$ is called a \textit{vertex splitting} of $I$ and $x_i$ is called a \textit{splitting vertex} of $I$.

The following lemma is proved in \cite[Theorem 3.6, Corollary 3.8]{MKA16} (see also \cite[Proposition 3]{CF2024}).

\begin{Lemma}\label{Lem:I(G)PI(H)}
	Let $G$ be a graph on the vertex set $[n]$ and assume that $I(G)$ has linear resolution. Then, up to a relabeling, 
	$$I(G)=x_nP+I(H)$$ is a vertex splitting, where $P= (x_j : x_jx_n\in I(G))$ and $H=G\setminus\{n\}$.
\end{Lemma}

Recall that for a monomial $u\in S$, the set $\supp(u)=\{i:\ x_i\ \textup{divides}\ u\}$ is called the \textit{support} of $u$, whereas if $I$ is a monomial ideal of $S$, the set
\[\supp(I)=\bigcup_{u\in\mathcal{G}(I)}\supp(u)\] 
is called the \textit{support} of $I$.

The next result slightly generalizes the characterization of the Cohen-Macaulay vertex splittable ideals proved in \cite[Theorem 2]{CF2024}. The proof is verbatim the same as that of \cite[Theorem 2]{CF2024}, therefore we omit it.

\begin{Theorem}\label{Thm:CF}
	Let $I,I_1,I_2\subset S$ be monomial ideals such that $I_2\subseteq I_1$, $i\notin\supp(I_2)$, $I=x_iI_1+I_2$ and $\mathcal{G}(I)=\mathcal{G}(x_iI_1)\cup\mathcal{G}(I_2)$. Furthermore, we assume that $I=x_iI_1+I_2$ is a Betti splitting. Then, the following statements are equivalent.
	\begin{enumerate}
		\item[\em (a)] $I$ is Cohen-Macaulay.
		\item[\em (b)] $I_1,I_2$ are Cohen-Macaulay and $\depth S/I_1=\depth S/(I_2,x_i)$.
	\end{enumerate} 
\end{Theorem}

In \cite{F1}, the concept of ${\bf t}$-spread strongly stable ideal was introduced. Let $d\ge2$ and ${\bf t}=(t_1,\dots,t_{d-1})\in\ZZ^{d-1}_{\ge0}$. Let $u=x_{i_1}\cdots x_{i_\ell}\in S$ with $1\le i_1\le\dots\le i_\ell\le n$ and $\ell\le d$. We say that $u$ is ${\bf t}$\textit{-spread} if $i_{j+1}-i_j\ge t_j$ for all $j=1,\dots,\ell-1$.

A monomial ideal $I\subset S$ is called ${\bf t}$\textit{-spread} if $\mathcal{G}(I)$ consists of ${\bf t}$-spread monomials. A ${\bf t}$-spread ideal $I\subset S$ is called \textit{${\bf t}$-spread strongly stable} if for all ${\bf t}$-spread monomials $u\in I$ and all $i<j$ such that $x_j$ divides $u$ and $x_i(u/x_j)$ is ${\bf t}$-spread, then $x_i(u/x_j)\in I$.

Let $u\in S$ be a ${\bf t}$-spread monomial. The smallest ${\bf t}$-spread ideal containing $u$ is called the \textit{principal ${\bf t}$-spread Borel ideal} generated by $u$, and is denoted by $B_{\bf t}(u)$. 

If $u=x_{n-t_{d-1}}x_{n-t_{d-2}-t_{d-1}}\cdots x_{n-t_1-\,\cdots \,- t_{d-1}}$, then $I=B_{\bf t}(u)$ is called the \textit{${\bf t}$-spread Veronese ideal of degree $d$} in $S$, and if $t_1=\dots=t_{d-1}=t$ for some $t$ then $B_{\bf t}(u)$ is called the \textit{uniform $t$-spread Veronese ideal of degree $d$} in $S$.

\begin{Proposition}\label{Prop:KP}
	Let $G\in\{K_n,P_n^c\}$ with $n\ge4$. Then $I(G)^{[k]}$ is bi-CM for all $1\le k\le \nu(G)$. Moreover, $\depth S/I(K_n)=1$ and $\depth S/I(P_n^c)=2$.
\end{Proposition}
\begin{proof}
	Note that $I(K_n)=B_{{\bf 1}}(x_{n-1}x_n)$ and $I(P_n^c)=B_{{\bf 2}}(x_{n-2}x_n)$, where ${\bf 1}=(1)$ and ${\bf 2}=(2)$. By \cite[Theorems 2.2 and 4.3]{CF}, $I(G)$ is bi-CM for $G\in\{K_n,P_n^c\}$. Let $\m=(x_1,\dots,x_n)$. Obviously $I(K_n)^{[k]}=\m^{[2k]}$ for all $k$, and this ideal is bi-CM. 
	
	Let $G=P_n^c$. In \cite[Example 1.4]{EHHM2022b}, the authors proved that $I(P_n^c)^{[2]}=\m^{[4]}$. Here we recover such a result in a simpler way. Let $x_ix_jx_kx_\ell\in\m^{[4]}$ be a monomial, with $1\le i<j<k<\ell\le n$. Then $x_ix_k,x_jx_\ell\in I(P_n^c)$ because they are ${\bf 2}$-spread monomials. Hence $u\in I(P_n^c)^{[2]}$ and consequently $I(P_n^c)^{[2]}=\m^{[4]}$, as desired. Next, by \cite[Proposition 1.3]{EHHM2022b}, $I(P_n^c)^{[k]}=\m^{[2k]}$ for all $k\ge2$, and such ideal is bi-CM. Finally, $I(G)^{[k]}$ is bi-CM for all $k$, with $G\in\{K_n,P_n^c\}$.
	
	The statement about the depth follows from \cite[Corollary 5.3]{F1} and the Auslander-Buchsbaum formula.
\end{proof}

\section{The classification}\label{sec2}

In this section we state and prove the main result of the paper, that is, the classification of all those graphs $G$ whose matching powers $I(G)^{[k]}$ are bi-CM, for all $1\le k\le\nu(G)$.

First, we note that the only graphs having up to three vertices are $K_2$, $P_3$ and $K_3$. Only $K_2$ and $K_3$ have the property that all their matching powers are bi-CM. Indeed, $P_3$ is not a Cohen-Macaulay graph.

The next definition will be useful in the sequel.
\begin{Definition}
	Let $I$ and $J$  be  two monomial ideals of the polynomial ring $S$. The monomial ideal defined as
	\[I*J=(uv\ :\ u\in\mathcal{G}(I),\,v\in\mathcal{G}(J),\,\supp(u)\cap\supp(v)=\emptyset),\]
	is called the \textit{matching product} of $I$ and $J$.
\end{Definition}

Let ${\bf 1}=(1,\dots,1)\in\ZZ^d_{\ge0}$. 
The next lemma will be crucial for our aim.

\begin{Lemma}\label{Lem:notCM}
	Let $d$ be an integer with $1<d<n$, and let $M_d$ be the set of all squarefree monomials of $S$ of degree $d$. Let $I\subset S$ be a squarefree monomial ideal generated by the set $M_d\setminus\{u\}$, for some $u\in M_d$. Then $I$ is not Cohen-Macaulay. 
\end{Lemma}
\begin{proof}
	After a suitable relabeling of the variables, we may assume that $$u=x_{n-d+1}x_{n-d+2}\cdots x_n.$$ Then, it is immediate to see that $I=B_{\bf1}(v)$, where $v=x_{n-d}(u/x_{n-d+1})$. It follows from \cite[Theorem 4.3]{CF} (or \cite[Proposition 1]{CF2024}) that $I$ is not Cohen-Macaulay.
\end{proof}

\begin{Theorem}\label{thm:main}
	Let $G$ be a graph on the vertex set $[n]$, with $n\ge4$. The following conditions are equivalent.
	\begin{enumerate}
		\item[\textup{(a)}] $I(G)^{[k]}$ is bi-CM for all $1\le k\le\nu(G)$.
		\item[\textup{(b)}] Either $G=K_n$ or $G\cong P_n^c$.
	\end{enumerate}
\end{Theorem}
\begin{proof}
	(b) $\Rightarrow$ (a): Follows from Proposition \ref{Prop:KP}.
	
	(a) $\Rightarrow$ (b): We proceed by induction on $n\ge4$. 
	
	Let $n=4$. It is easily checked that the only bi-CM graphs on 4 non-isolated vertices are either complete graphs or isomorphic to complements of a path on four vertices. Indeed, up to isomorphism the following seven graphs are the only graphs on 4 non-isolated vertices\bigskip
	
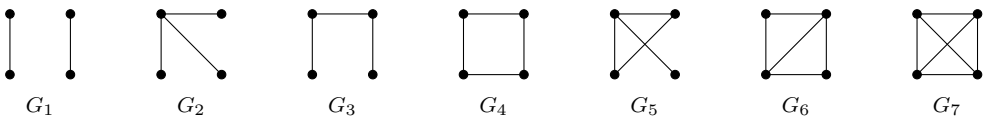
\begin{center}
		\begin{tikzpicture}[scale=0.8]
			\draw[-] (1,0) -- (1,1);
			\draw[-] (2,0) -- (2,1);
			\draw[-] (3.5,0) -- (3.5,1) -- (4.5,1);
			\draw[-] (3.5,1) -- (4.5,0);
			\draw[-] (6,0) -- (6,1) -- (7,1) -- (7,0);
			\draw[-] (8.5,0) -- (8.5,1) -- (9.5,1) -- (9.5,0) -- (8.5,0);
			\draw[-] (11,0) -- (12,1) -- (11,1) -- (11,0);
			\draw[-] (11,1) -- (12,0);
			\draw[-] (13.5,0) -- (13.5,1) -- (14.5,1) -- (14.5,0) -- (13.5,0) -- (14.5,1);
			\draw[-] (16,0) -- (16,1) -- (17,1) -- (17,0) -- (16,0) -- (17,1);
			\draw[-] (16,1) -- (17,0);
			\filldraw[black] (1,0) circle (2pt);
			\filldraw[black] (2,0) circle (2pt);
			\filldraw[black] (1,1) circle (2pt);
			\filldraw[black] (2,1) circle (2pt);
			\filldraw[black] (1.5,-0.2) node[below]{{\tiny$G_1$}};
			\filldraw[black] (3.5,0) circle (2pt);
			\filldraw[black] (4.5,0) circle (2pt);
			\filldraw[black] (3.5,1) circle (2pt);
			\filldraw[black] (4.5,1) circle (2pt);
			\filldraw[black] (4,-0.2) node[below]{{\tiny$G_2$}};
			\filldraw[black] (6,0) circle (2pt);
			\filldraw[black] (7,0) circle (2pt);
			\filldraw[black] (6,1) circle (2pt);
			\filldraw[black] (7,1) circle (2pt);
			\filldraw[black] (6.5,-0.2) node[below]{{\tiny$G_3$}};
			\filldraw[black] (8.5,0) circle (2pt);
			\filldraw[black] (9.5,0) circle (2pt);
			\filldraw[black] (8.5,1) circle (2pt);
			\filldraw[black] (9.5,1) circle (2pt);
			\filldraw[black] (9,-0.2) node[below]{{\tiny$G_4$}};
			\filldraw[black] (11,0) circle (2pt);
			\filldraw[black] (12,0) circle (2pt);
			\filldraw[black] (11,1) circle (2pt);
			\filldraw[black] (12,1) circle (2pt);
			\filldraw[black] (11.5,-0.2) node[below]{{\tiny$G_5$}};
			\filldraw[black] (13.5,0) circle (2pt);
			\filldraw[black] (14.5,0) circle (2pt);
			\filldraw[black] (13.5,1) circle (2pt);
			\filldraw[black] (14.5,1) circle (2pt);
			\filldraw[black] (14,-0.2) node[below]{{\tiny$G_6$}};
			\filldraw[black] (16,0) circle (2pt);
			\filldraw[black] (17,0) circle (2pt);
			\filldraw[black] (16,1) circle (2pt);
			\filldraw[black] (17,1) circle (2pt);
			\filldraw[black] (16.5,-0.2) node[below]{{\tiny$G_7$}};
		\end{tikzpicture}
	\end{center}
	It is easily seen that the only bi-CM graphs among these seven graphs are $G_3$ and $G_7$. In fact, $G_1$ is a Cohen-Macaulay graph which is not connected (and so $I(G_1)$ can not have linear resolution), whereas $G_2, G_4, G_5$ and $G_6$ are not Cohen-Macaulay. On the other hand, $G_3\cong P_4^c$ and $G_7=K_4$ and, by Proposition \ref{Prop:KP}, the statement in the case $n=4$ follows.
	
	Now let $n>4$ and let $G$ be a graph on the vertex set $[n]$ such that $I(G)^{[k]}$ is bi-Cohen-Macaulay for all $1\le k\le\nu(G)$. In particular, $I(G)$ has a linear resolution. Up to a suitable relabeling, by Lemma \ref{Lem:I(G)PI(H)}, 
	$$I(G)=x_nP+I(H),$$ 
	where $H=G\setminus\{n\}$ and $P=(x_j:x_jx_n\in I(G))$ contains $I(H)$. Using the fact that $(x_nP)^{[\ell]}=0$ for $\ell\ge2$, we can note that
	\begin{equation}\label{eq:I(G)BS}
		I(G)^{[k]}\ =\ (x_nP)*I(H)^{[k-1]}+I(H)^{[k]}\ = \ x_n(P*I(H)^{[k-1]})+I(H)^{[k]},
	\end{equation}
	where $*$ is the matching product previously defined.
	
	By \cite[Corollary 3.1]{FShort}, $PI(H)^{k-1}$ has linear resolution. Notice that $P*I(H)^{[k-1]}$ is the squarefree part of $PI(H)^{k-1}$. Hence, by \cite[Lemma 1.2]{EHHM2022a}, $P*I(H)^{[k-1]}$ has also linear resolution. Since both the ideals $(x_nP)*I(H)^{[k-1]}$ and $I(H)^{[k]}$ have linear resolution, and $\mathcal{G}(I(G)^{[k]})=\mathcal{G}((x_nP)*I(H)^{[k-1]})\cup\mathcal{G}(I(H)^{[k]})$, by \cite[Corollary 2.4]{FHT} we have that (\ref{eq:I(G)BS}) is a Betti splitting. Notice that $I(H)^{[k]}\subseteq P*I(H)^{[k-1]}$ and that $n\notin\supp(I(H)^{[k]})$. So, Theorem \ref{Thm:CF} implies that $I(H)^{[k]}$ is Cohen-Macaulay for all $1\le k\le\nu(G)$. By Lemma \ref{Lem:res}, $I(H)^{[k]}$ has linear resolution for all $1\le k\le\nu(H)$. Hence, $I(H)^{[k]}$ is bi-CM for all $k$. By inductive hypothesis and after a suitable relabeling of the vertices, either $H=K_{m}$ or $H\cong P_{m}^c$ for some $m\le n-1$.
	
	Since $I(G)=x_nP+I(H)$ is a Betti splitting and $I(G)$ is Cohen-Macaulay, Theorem \ref{Thm:CF} implies that
	$$
	\depth\frac{S}{P}\ =\ \depth\frac{S}{(I(G),x_n)}\ =\ \depth\frac{K[x_1,\dots,x_{n-1}]}{I(H)}.
	$$
	
	Consequently,
	\begin{equation}\label{eq:muP}
		\mu(P)\ =\ n-\depth\frac{K[x_1,\dots,x_{n-1}]}{I(H)},
	\end{equation}
	where $\mu(P)$ is the cardinality of a minimal system of generators of $P$.
	
	Now, we distinguish the two possible cases, that is, $H=K_{m}$ and $H\cong P_{m}^c$. \smallskip
	\par\noindent
	\textbf{Case 1.} Assume $H=K_{m}$. By Proposition \ref{Prop:KP} we have
	$$
	\depth\frac{K[x_1,\dots,x_{n-1}]}{I(H)}\ =\ \depth\frac{K[x_1,\dots,x_m]}{I(K_m)}+(n-1-m)\ =\ n-m.
	$$
	From equation (\ref{eq:muP}) we obtain $\mu(P)=m$. Since
	$$
	[n]\ =\ V(G)\ =\ V(H)\cup\{n\}\cup\{i\ :\ x_i\in P\}\ =\ [m]\cup\{n\},
	$$
	and $I(H)=I(K_{m})\subset P$, we deduce that, up to a relabeling, $P=(x_1,\dots,x_{m-1},x_p)$ for some $m\le p\le n-1$ and either $m=n-1$ or $m=n-2$. If $m=n-1$ then $P=(x_1,\dots,x_{n-1})$ and $G$ is the complete graph, as desired.
	
	If otherwise $m=n-2$, then we must have $p=n-1$ and $P=(x_1,\dots,x_{n-3},x_{n-1})$. It is immediate to see that
	$$
	I(G)^{[2]}=x_nP*I(K_{n-2})+I(K_{n-2})^{[2]}=x_n(x_1,\dots,x_{n-1})^{[3]}+(x_1,\dots,x_{n-2})^{[4]},
	$$
	and by the argument after equation (\ref{eq:I(G)BS}) this is a Betti splitting. Now, if $n=5$, then $I(G)^{[2]}=x_5(x_1,\dots,x_{4})^{[3]}$ is not Cohen-Macaulay, because it is not unmixed. Otherwise, let $n\ge6$, then $(x_1,\dots,x_{n-2})^{[4]}\ne(0)$. By Theorem \ref{Thm:CF}, since $I(G)^{[2]}$ is Cohen-Macaulay, we should have
	$$
	\depth\frac{S}{(x_1,\dots,x_{n-1})^{[3]}}=\depth\frac{S}{((x_1,\dots,x_{n-2})^{[4]},x_n)}.
	$$
	However, by \cite[Lemma 2]{CF2024} the first depth is equal to 3, while the second depth is equal to 4. Hence, this case does not occur.\smallskip
	\par\noindent
	\textbf{Case 2.} Assume $H\cong P_{m}^c$. By Proposition \ref{Prop:KP}, we have
	$$
	\depth\frac{K[x_1,\dots,x_{n-1}]}{I(H)}\ =\ \depth\frac{K[x_1,\dots,x_m]}{I(P_m^c)}+(n-1-m)\ =\ n-m+1.
	$$
	
	Then, by equation (\ref{eq:muP}) we get that
	\begin{equation}\label{eq:muP1}
		\mu(P)\ =\ m-1.
	\end{equation}
	Since $[n]=V(H)\cup\{i:x_i\in P\}\cup\{n\}$ and $V(H)=[m]$, we have $x_{m+1},\dots,x_{n-1}\in P$. Write $P=(Q,x_{m+1},\dots,x_{n-1})$, where $Q\subseteq(x_1,\dots,x_m)$ is a monomial prime ideal. Since $x_{m+1},\dots,x_{n-1}\notin I(H)$ and $I(H)\subseteq P$, it follows that $I(H)\subseteq Q$. Using again Proposition \ref{Prop:KP} we have $\height\,I(H)=m-2$, and so $\mu(Q)\ge m-2$. Consequently $\mu(P)=\mu(Q)+n-m-1\ge n-3$. Taking into account (\ref{eq:muP1}) we have $m\ge n-2$. Since, by definition, $P\subseteq(x_1,\dots,x_{n-1})$, then we have $m\le n-1$. Hence either $m=n-2$ or $m=n-1$. We distinguish the two following cases.\medskip
	
	\noindent\textbf{Case 2.1.} Let $m=n-1$. Thus $H=P_{n-1}^c$. For $1\le i\le n-1$, let $P_i$ be the monomial prime ideal generated by the set of variables $\{x_1,\dots,x_{n-1}\}\setminus\{x_i\}$. Since $P\subseteq(x_1,\dots,x_{n-1})$ and $\mu(P)=n-2$, we see that $P=P_i$ for some $1\le i\le n-1$. If $P=P_1$ or $P=P_{n-1}$, then $G^c$ is a path, and by Proposition \ref{Prop:KP}, $I(G)^{[k]}$ is indeed Cohen-Macaulay for all $1\le k\le\nu(G)$ and (b) holds in this case.
	
	So, it is enough to show that $P$ can not be equal to $P_i$ for some $2\le i\le n-2$. Suppose that $P=P_i$ for some $2\le i\le n-2$. 
	
	We claim that
	\begin{equation}\label{eq:G-Claim}
		\mathcal{G}(I(G)^{[2]})\ =\ \mathcal{G}(\m^{[4]})\setminus\{x_{i-1}x_ix_{i+1}x_n\},
	\end{equation}
	where $\m=(x_1,\dots,x_n)$. Then Lemma \ref{Lem:notCM} shows that $I(G)^{[2]}$ is not Cohen-Macaulay, which contradicts the assumption.
	
	Let us show that $\mathcal{G}(I(G)^{[2]})=\mathcal{G}(\m^{[4]})\setminus\{x_{i-1}x_ix_{i+1}x_n\}$ if $P=P_i$ for some $2\le i\le n-2$. 
	
	Since $I(G)=I(H)+x_nP$, we have that
	\begin{equation}\label{eq:I(G)[2]}
		\begin{aligned}
			I(G)^{[2]}\ &=\ I(H)^{[2]}+I(H)*(x_nP)+(x_nP)^{[2]}\\
			&=\ (x_1,\dots,x_{n-1})^{[4]}+x_n(P*I(H)),
		\end{aligned}
	\end{equation}
	where we have used the fact that $(x_nP)^{[2]}=(0)$.
	
	By (\ref{eq:I(G)[2]}), all monomials of $\mathcal{G}(\m^{[4]})$ which are not divided by $x_n$ belong to $\mathcal{G}(I(G)^{[2]})$. Let $u=x_{j}x_{k}x_{\ell}x_n$ be a squarefree monomial divided by $x_n$, with $1\le j<k<\ell<n$. Next, we show that $u\in\mathcal{G}(I(G)^{[2]})$ if and only if $u\ne x_{i-1}x_ix_{i+1}x_n$. This will prove equation (\ref{eq:G-Claim}).
	
	If none of the integers $j,k,\ell$ is equal to $i$, then $\{j,\ell\}\in E(G)$ because $\ell\ge j+2$ and $I(H)=I(P_{n-1}^c)\subset I(G)$. Moreover $\{k,n\}\in E(G)$ because $x_k\in P$. Then, we have $u=(x_jx_\ell)(x_kx_n)\in I(G)^{[2]}$, as desired.
	
	Suppose now that one of the integers $j,k,\ell$ is equal to $i$.
	
	If $j=i$ or $\ell=i$, then in both cases $\{j,\ell\}\in E(G)$ because $\ell\ge j+2$. As before, $x_k\in P$ since $k\ne i$ and so $\{k,n\}\in E(G)$. Hence $u=(x_jx_\ell)(x_kx_n)\in I(G)^{[2]}$, once again.
	
	Let $k=i$. Then $1\le j\le i-1$ and $i+1\le\ell\le n-1$.
	
	Suppose that $j<i-1$. Then $\{j,i\}\in E(H)\subset E(G)$ because $i\ge j+2$. Since $\ell\ne i$, we have $x_\ell\in P$. Hence $\{\ell,n\}\in E(G)$, and so $u=(x_jx_i)(x_\ell x_n)\in I(G)^{[2]}$.
	
	Similarly, if $\ell>i+1$, then $\{i,\ell\}\in E(H)\subset E(G)$ because $\ell\ge i+2$. Moreover $x_j\in P$ since $j\ne i$, and so $\{j,n\}\in E(G)$. Hence $u=(x_ix_\ell)(x_jx_n)\in I(G)$.
	
	Finally, assume that $j=i-1$ and $\ell=i+1$. We show that $u\notin I(G)^{[2]}$. Notice that $\{i,n\}\notin E(G)$ since $x_i\notin P$. Hence, $u$ belongs to $I(G)^{[2]}$, if and only if, either $\{i-1,i\},\{i+1,n\}\in E(G)$, or $\{i,i+1\},\{i-1,n\}\in E(G)$. Notice that $i+1\le n-1$ and the restriction of $G$ to the vertex set $[n-1]$ is $H=P_{n-1}^c$. Hence $\{i-1,i\},\{i,i+1\}\notin E(G)$, and so $u\notin I(G)^{[2]}$, as claimed.\medskip
	
	\noindent\textbf{Case 2.2.} Let $m=n-2$. We will show that this case can never occur, and this will conclude the proof.
	
	Let $n$ be odd, say $n=2k+1$ for some $k\ge2$. Since $I(G)^{[\nu(G)]}$ is Cohen-Macaulay by assumption, by \cite[Theorem 1.8(b)]{FM} we have that $\nu(G)=k$ and $I(G)^{[k]}=\m^{[2k]}$. In particular, $u=x_1\cdots x_{2k}\in\mathcal{G}(I(G)^{[k]})$. Notice that
	$$
	I(G)^{[k]}=I(P_{2k-1}^c)^{[k]}+I(P_{2k-1}^c)^{[k-1]}*(x_nP)=I(P_{2k-1}^c)^{[k-1]}*(x_nP),
	$$
	because $\nu(P^c_{2k-1})=k-1$. Hence all minimal monomial generators of $I(G)^{[k]}$ are divided by $x_n$, and so $u\notin\mathcal{G}(I(G)^{[k]})$. A contradiction.
	
	Now, let $n$ be even, say $n=2k$ with $k\ge3$. Then $k-1\ge2$, $\nu(G)=k$ and by the proof of Proposition \ref{Prop:KP} we have
	\begin{align*}
		I(G)^{[k-1]}\ &=\ I(P^c_{2k-2})^{[k-1]}+(x_nP)*I(P^c_{2k-2})^{[k-2]}\\ &=\ (x_1\cdots x_{2k-2})+x_n(P*I(P^c_{2k-2})^{[k-2]}).
	\end{align*}
	Since $\mu(P)=n-3$, we can find $1\le i\le n-2$ with $x_i\notin P$. We claim that $Q=(x_i,x_n)$ is a minimal prime ideal of $I(G)^{[k-1]}$. Indeed, from the above decomposition it is clear that $I(G)^{[k-1]}\subseteq Q$. Notice that $(x_n)$ does not contain $I(G)^{[k-1]}$ because $x_n$ does not divide $x_1\cdots x_{2k-2}$ and $(x_i)$ does not contain $I(G)^{[k-1]}$ because $x_i$ does not divide $x_nx_{n-1}u$ for some $u\in\mathcal{G}(I(P^c_{2k-2})^{[k-2]})$ with $x_i$ not dividing $u$. It is possible to find such a monomial $u$ because $\nu(P^c_{2k-2})=k-1>k-2$. 
	
	Hence $Q\in\Ass\,I(G)^{[k-1]}$ and this implies that $\dim S/I(G)^{[k-1]}=2k-2$. Since $I(G)^{[k]}$ is Cohen-Macaulay by assumption and $n=2k$, then \cite[Theorem 1.8(b)]{FM} implies that $G$ has a perfect matching. Consequently \cite[Theorem 2.2(c)]{FM} implies that $G$ is a Cohen-Macaulay forest. This is easily seen to be impossible. Indeed, for $k\ge4$ we have that $\{1,3,5\}$ is a clique in $H=P_{2k-2}^c$ because $2k-2\ge 6$ since $k\ge4$. So $G$ is not even a forest. Whereas, for $k=3$, we have $H=P^c_4$ and so $I(G)=x_6(x_i,x_j,x_5)+(x_1x_3,x_1x_4,x_2x_4)$ for some integers $1\le i<j\le 4$. It is easily seen that for all possible choices of $i,j$, the graph $G$ contains an induced cycle and so is not even a forest. We reach a contradiction in any case, as desired.
\end{proof}

We expect that any uniform $t$-spread Veronese ideal of degree $d\ge t$ has the property that all its squarefree powers are bi-CM.

\bigskip\bigskip
\noindent\textbf{Acknowledgment.}
A. Ficarra was partly supported by the Grant JDC2023-051705-I funded by
MICIU/AEI/10.13039/501100011033 and by the FSE+. Moreover, both the authors acknowledge support of the GNSAGA of INdAM (Italy).


\begin{thebibliography}{99}
	
	\bibitem{BHZ} Bigdeli, M.,  Herzog, J.,  Zaare-Nahandi, R. (2018). On the index of powers of edge ideals. \textit{Comm. Algebra}. 46: 1080--1095. 
	
	\bibitem{CF} Crupi, M., Ficarra, A. (2023). A note on minimal resolutions of vector--spread Borel ideals. \textit{Analele Stiintifice ale Universitatii Ovidius Constanta, Seria Matematica}. 31(2): 71--84.
	
	\bibitem{CF2024} Crupi, M., Ficarra, A. (2024). Cohen--Macaulayness of vertex splittable monomial ideals. \textit{Mathematics}. 12(6): 912. Available at https://doi.org/10.3390/math12060912.
	
	\bibitem{CFL3} Crupi, M., Ficarra, A., Lax, E. (2025). Matchings, Squarefree Powers and Betti splittings. \textit{Illinois J. Math.} 69(2), 353-- 372. 
	
	\bibitem{DRS24} Das,  K. K., Roy, A., Saha, K. (2024). Square-free powers of Cohen-Macaulay forests, cycles, and whiskered cycles. Available at \url{https://arxiv.org/abs/2409.06021}.
	
	\bibitem{DRS25} Das,  K. K., Roy, A., Saha, K. (2025). Square-free powers of Cohen-Macaulay simplicial forests. Available at \url{https://arxiv.org/abs/2502.18396}.
	
	\bibitem{EF} Erey, N., Ficarra, A. (2023). Matching Powers of monomial ideals and edge ideals of weighted oriented graphs. \textit{J. Algebra Appl.}. \url{https://doi.org/10.1142/S0219498826501185}. 
	
	\bibitem{EF1} Erey, N., Ficarra, A. (2024). Forests whose matching powers are linear. Available at \url{https://arxiv.org/abs/2403.17797}.
	
	\bibitem{EH2021} Erey, N., Hibi, T. (2021). Squarefree powers of edge ideals of forests. \textit{Electron. J. Combin.:}. 28(2): P2.32.
	
	\bibitem{EHHM2022a} Erey, N., Herzog, J., Hibi, T., Saeedi Madani S. (2022). Matchings and squarefree powers of edge ideals. \textit{J. Comb. Theory Series. A}. 188: 105585
	
	\bibitem{EHHM2022b} Erey, N., Herzog, J., Hibi, T., Saeedi Madani S. (2024).  The normalized depth function of squarefree powers. \textit{Collect. Math.}. 75:409--423.
	
	\bibitem{FPack2} Ficarra, A. (2023). Matching Powers: Macaulay2 Package. Available at  \url{https://arxiv.org/abs/2312.13007}.
	
	\bibitem{F1} Ficarra, A. (2023). Vector-spread monomial ideals and Eliahou--Kervaire type resolutions. \textit{J. Algebra}. 615: 170--204.
	
	\bibitem{FShort} Ficarra, A. (2024). A new proof of the Herzog-Hibi-Zheng theorem. Available at \url{https://arxiv.org/abs/2409.15853}
	
	\bibitem{FHH2022} Ficarra, A., Herzog, J., Hibi, T. (2023). Behaviour of the normalized depth function. \textit{Electron. J. Comb.}. 30 (2): P2.31.
	
	\bibitem{FM} Ficarra, A., Moradi, S. (2024). Monomial ideals whose all matching powers are Cohen-Macaulay. Available at \url{https://arxiv.org/abs/2410.01666}
	
	\bibitem{FV2005} Fl\o ystad, G. , Vatne, J. E. (2005). (Bi-)Cohen–Macaulay simplicial complexes and their associated coherent sheaves. \textit{Comm. Algebra}. 33: 3121--3136.
	
	\bibitem{FHT} Francisco, C.~A.~, Hà, H.~T.~, Van Tuyl, A. (2009). Splittings of monomial ideals. \textit{Proc. Amer. Math. Soc.} 137(10): 3271--3282.
	
	\bibitem{HHBook} Herzog J., Hibi T. (2011). \textit{Monomial ideals}, \textit{Graduate texts in Mathematics}, Vol. 260, London, UK: Springer--Verlag. https://doi.org/10.1007/978-0-85729-106-6 
	
	\bibitem{HR16} Herzog, J., Rahimi, A. (2016). Bi-Cohen–Macaulay graphs. \textit{Electron. J. Combin.} 23: \#P1.1.
	
	\bibitem{KNQ24} Kamberi, E., Navarra, F., Qureshi, A. A. (2024). On squarefree powers of simplicial trees. Available at  \url{https://arxiv.org/abs/2406.13670}.
	
	\bibitem{MKA16} Moradi, S.,  Khosh-Ahang, F. (2016). On vertex decomposable simplicial complexes and their Alexander duals. \textit{Math. Scand.} 118(1): 43--56.
	
	\bibitem{SASF2024} Seyed Fakhari, S. A. (2024). An increasing normalized depth function. \textit{J. Commut. Algebra}. In Press. Available at \url{arxiv.org/abs/2309.13892}.
	
	\bibitem{SASF2022} Seyed Fakhari, S. A. (2024). On the Castelnuovo-Mumford regularity of squarefree powers of edge ideals. \textit{J. Pure Appl. Algebra}. 228(3):107488.
	
	\bibitem{SASF2023} Seyed Fakhari, S. A. (2025). On the Regularity of squarefree part of symbolic powers of edge ideals. \textit{J. Algebra}. 665:103--130.
	
	\bibitem{SASF2024b} Seyed Fakhari, S. A. (2024). Matchings and Castelnuovo-Mumford regularity of squarefree powers of edge ideals, preprint.
	
	\bibitem{RVbook} Villarreal, R. H. (2025).  \textit{Monomial Algebras}, Second Edition. \textit{Monographs and Research Notes in Mathematics}. CRC Press, Boca Raton, FL.
\end{thebibliography}
\end{document}